\documentclass[10pt,reqno]{amsart} 

\usepackage{amsthm,amsfonts,amscd,amsmath}
\usepackage{hyperref} 
\usepackage{url}
\usepackage{xcolor}
\hypersetup{
    colorlinks,
    linkcolor={red!50!black},
    citecolor={blue!50!black},
    urlcolor={blue!80!black}
}

\usepackage[normalem]{ulem}
\usepackage{graphicx} 
\usepackage{datetime} 
\usepackage{arydshln} 

\newcommand{\cf}{\textit{cf.\ }} 
\newcommand{\Iverson}[1]{\ensuremath{\left[#1\right]_{\delta}}} 

\title[Arithmetic Functions Generated by Lambert Series Factorizations]{
       Generating Special Arithmetic Functions by Lambert Series Factorizations
} 
\author[Mircea Merca and Maxie D. Schmidt]{
        Mircea Merca \\ 
        Academy of Romanian Scientists \\ 
        Splaiul Independentei 54, Bucharest, 050094 Romania \\ 
        \href{mailto:mircea.merca@profinfo.edu.ro}{mircea.merca@profinfo.edu.ro} \\ 
        \\ 
        Maxie D. Schmidt \\ 
        School of Mathematics \\ 
        Georgia Institute of Technology \\ 
        Atlanta, GA 30332 USA \\ 
        \href{mailto:maxieds@gmail.com}{maxieds@gmail.com}
} 
\date{\today} 

\usepackage{setspace}

\allowdisplaybreaks 

\theoremstyle{plain} 
\newtheorem{theorem}{Theorem}
\newtheorem{conjecture}[theorem]{Conjecture}
\newtheorem{prop}[theorem]{Proposition}

\newtheorem{cor}[theorem]{Corollary}
\numberwithin{theorem}{section}

\theoremstyle{definition} 

\newtheorem{remark}[theorem]{Remark}

\makeatletter
\renewcommand\@biblabel[1]{#1.}
\makeatother
\usepackage[tableposition=top]{caption}
\begin{document} 

\begin{abstract} 
We summarize the known useful and interesting results and formulas 
we have discovered so far in this collaborative article summarizing 
results from two related articles by Merca and Schmidt arriving at 
related so-termed Lambert series factorization theorems. 
We unify the matrix representations that 
underlie two of our separate papers, and which commonly arise in identities 
involving partition functions and other functions generated by Lambert series. 
We provide a number of properties and conjectures related to the inverse matrix 
entries defined in Schmidt's article and the Euler partition function $p(n)$ 
which we prove through our new results unifying the expansions of the 
Lambert series factorization theorems within this article. 
\end{abstract}

\maketitle

\section{Introduction} 

\subsection{Lambert series factorization theorems} 

We consider recurrence relations and matrix equations 
related to \emph{Lambert series} expansions of the form 
\cite[\S 27.7]{NISTHB} \cite[\S 17.10]{HARDYANDWRIGHT} 
\begin{align}
\label{eqn_LambertSeriesfb_def} 
\sum_{n \geq 1} \frac{a_n q^n}{1-q^n} & = \sum_{m \geq 1} b_m q^m,\ |q| < 1, 
\end{align} 
for prescribed arithmetic functions 
$a: \mathbb{Z}^{+} \rightarrow \mathbb{C}$ and 
$b: \mathbb{Z}^{+} \rightarrow \mathbb{C}$ where $b_m = \sum_{d | m} a_d$. 
There are many well-known Lambert series for special arithmetic functions of the 
form in \eqref{eqn_LambertSeriesfb_def}. 
Examples include the following series where $\mu(n)$ denotes the 
\emph{M\"obius function}, $\phi(n)$ denotes \emph{Euler's phi function}, 
$\sigma_{\alpha}(n)$ denotes the generalized \emph{sum of divisors function}, 
$\lambda(n)$ denotes \emph{Liouville's function}, $\Lambda(n)$ denotes 
\emph{von Mangoldt's function}, $\omega(n)$ defines the number of 
distinct primes dividing $n$, and $J_t(n)$ is 
\emph{Jordan's totient function} for a fixed $t \in \mathbb{C}$ 
\cite[\S 27.6 -- \S 27.7]{NISTHB}\footnote{ 
     \underline{\emph{Notation}}: 
     \emph{Iverson's convention} compactly specifies 
     boolean-valued conditions and is equivalent to the 
     \emph{Kronecker delta function}, $\delta_{i,j}$, as 
     $\Iverson{n = k} \equiv \delta_{n,k}$. 
     Similarly, $\Iverson{\mathtt{cond = True}} \equiv 
                 \delta_{\mathtt{cond}, \mathtt{True}}$ 
     in the remainder of the article. 
}: 
\begin{align} 
\label{eqn_WellKnown_LamberSeries_Examples} 
\sum_{n \geq 1} \frac{\mu(n) q^n}{1-q^n} & = q, && 
     (a_n, b_n) := (\mu(n), \Iverson{n = 1}) \\ 
\notag
\sum_{n \geq 1} \frac{\phi(n) q^n}{1-q^n} & = \frac{q}{(1-q)^2}, && 
     (a_n, b_n) := (\phi(n), n) \\ 
\notag
\sum_{n \geq 1} \frac{n^{\alpha} q^n}{1-q^n} & =  
     \sum_{m \geq 1} \sigma_{\alpha}(n) q^n, && 
     (a_n, b_n) := (n^{\alpha}, \sigma_{\alpha}(n)) \\ 
\notag
\sum_{n \geq 1} \frac{\lambda(n) q^n}{1-q^n} & = \sum_{m \geq 1} q^{m^2}, && 
     (a_n, b_n) := (\lambda(n), \Iverson{\text{$n$ is a positive square}}) \\ 
\notag 
\sum_{n \geq 1} \frac{\Lambda(n) q^n}{1-q^n} & = \sum_{m \geq 1} \log(m) q^m, && 
     (a_n, b_n) := (\Lambda(n), \log n) \\ 
\notag 
\sum_{n \geq 1} \frac{|\mu(n)| q^n}{1-q^n} & = \sum_{m \geq 1} 2^{\omega(m)} q^m, && 
     (a_n, b_n) := (|\mu(n)|, 2^{\omega(n)}) \\ 
\notag 
\sum_{n \geq 1} \frac{J_t(n) q^n}{1-q^n} & = \sum_{m \geq 1} m^t q^m, && 
     (a_n, b_n) := (J_t(n), n^t). 
\end{align}
In this article, our new results and conjectures extend and unify the related 
Lambert series factorization theorems considered in two separate contexts in 
the references \cite{MERCA-LSFACTTHM,SCHMIDT-LSFACTTHM}. 
In particular, in \cite{MERCA-CIRCND} Merca notes that 
\begin{align*} 
\sum_{n \geq 1} \frac{q^n}{1 \pm q^n} & = \frac{1}{(\mp q; q)_{\infty}} 
     \sum_{n \geq 1} \left(s_o(n) \pm s_e(n)\right) q^n, 
\end{align*} 
where $s_o(n)$ and $s_e(n)$ respectively denote the number of parts in all 
partitions of $n$ into an odd (even) number of distinct parts. 
More generally, Merca \cite{MERCA-LSFACTTHM} proves that 
\begin{align} 
\label{eqn_Merca_LSFactorizationThm} 
\sum_{n \geq 1} \frac{a_n q^n}{1 \pm q^n} & = 
     \frac{1}{(\mp q; q)_{\infty}} \sum_{n \geq 1} \left(\sum_{k=1}^n 
     \left(s_o(n, k) \pm s_e(n, k)\right) a_k\right) q^n, 
\end{align} 
where $s_o(n, k)$ and $s_e(n, k)$ are respectively the number of $k$'s in all 
partitions of $n$ into an odd (even) number of distinct parts. 

\begin{table}[ht!] 

\caption{The bottom row sequences in the matrices, $A_n^{-1}$, in the 
         definition of \eqref{eqn_AnInv_BlockMatrixRecForm} on page 
         \pageref{eqn_AnInv_BlockMatrixRecForm} for $2 \leq n \leq 18$.}
\label{table_matrix_An_last_row_seqs} 

\begin{center}
\begin{equation*}
\boxed{
\begin{array}{|c||l|} \hline
n & r_{n,n-1}, r_{n,n-2}, \ldots, r_{n,1} \\ \hline 
2 & 1 \\ 
3 & 1, 1 \\ 
4 & 2, 1, 1 \\ 
5 & 4, 3, 2, 1 \\ 
6 & 5, 3, 2, 2, 1 \\ 
7 & 10, 7, 5, 3, 2, 1 \\ 
8 & 12, 9, 6, 4, 3, 2, 1 \\ 
9 & 20, 14, 10, 7, 5, 3, 2, 1 \\ 
10 & 25, 18, 13, 10, 6, 5, 3, 2, 1 \\ 
11 & 41, 30, 22, 15, 11, 7, 5, 3, 2, 1 \\ 
12 & 47, 36, 26, 19, 14, 10, 7, 5, 3, 2, 1 \\ 
13 & 76, 56, 42, 30, 22, 15, 11, 7, 5, 3, 2, 1, 1 \\ 
14 & 90, 69, 51, 39, 28, 21, 14, 11, 7, 5, 3, 2, 1, 1  \\ 
15 & 129, 97, 74, 55, 41, 30, 22, 15, 11, 7, 5, 3, 2, 1, 1  \\ 
16 & 161, 124, 94, 72, 53, 40, 29, 21, 15, 11, 7, 5, 3, 2, 1, 1  \\ 
17 & 230, 176, 135, 101, 77, 56, 42, 30, 22, 15, 11, 7, 5, 3, 2, 1, 1  \\ 
18 & 270, 212, 163, 126, 95, 73, 54, 41, 29, 22, 15, 11, 7, 5, 3, 2, 1, 1  \\ 
\hline
\end{array} 
}
\end{equation*} 
\end{center} 

\end{table} 

\subsection{Matrix equations for the arithmetic functions generated by Lambert series} 

We then define the invertible $n \times n$ square matrices, $A_n$, as in 
Schmidt's article according to the 
convention from Merca's article as \cite[\cf \S 1.2]{SCHMIDT-LSFACTTHM} 
\begin{align} 
\label{eqn_An_SquareMatrix_def} 
A_n & := \left(s_e(i, j) - s_o(i, j)\right)_{1 \leq i,j \leq n}, 
\end{align} 
where the entries, $s_{i,j} := s_e(i, j) - s_o(i, j)$, of these matrices 
are generated by \cite[Cor.\ 4.3]{MERCA-LSFACTTHM} 
\begin{align*} 
s_{i,j} & = s_e(i, j) - s_o(i, j) = [q^i] \frac{q^j}{1-q^j} (q; q)_{\infty}. 
\end{align*} 
We then have formulas for the Lambert series arithmetic functions, $a_n$, in 
\eqref{eqn_LambertSeriesfb_def} and in the special cases from 
\eqref{eqn_WellKnown_LamberSeries_Examples} for all $n \geq 1$ given by 
\begin{align} 
\label{eqn_fn_matrix_eqn}
\begin{bmatrix} a_1 \\ a_2 \\ \vdots \\ a_n \end{bmatrix} & = 
     A_n^{-1} \left( 
     \underset{:= B_{b,m}}{\underbrace{
     b_{m+1} - \sum_{s = \pm 1} \sum_{k=1}^{\lfloor \frac{\sqrt{24m+1}-s}{6} \rfloor} 
     (-1)^{k+1} b_{m+1-k(3k+s)/2}}} 
     \right)_{0 \leq m < n}. 
\end{align} 
In general, for all $n \geq 2$ we have recursive formulas for the inverse matrices 
defined by \eqref{eqn_An_SquareMatrix_def} expanded in the form of 
\begin{align} 
\label{eqn_AnInv_BlockMatrixRecForm} 
A_{n+1}^{-1} & = 
     \left[
     \begin{array}{c|c} 
     A_n^{-1} & \mathbf{0} \\ \hdashline 
     r_{n+1,n}, \ldots, r_{n+1, 1} & 1 
     \end{array}
     \right], 
\end{align} 
where the first several special cases of the sequences, 
{\small$\{r_{n,n}, r_{n,n-1}, \ldots, r_{n,1}\}$}, are given as in 
\cite{SCHMIDT-LSFACTTHM} by Table \ref{table_matrix_An_last_row_seqs}. 

Within this article we focus on the properties of the entries, $s_{i,j}^{(-1)}$, 
of the inverse matrices, $A_n^{-1}$, defined by \eqref{eqn_An_SquareMatrix_def}. 
We prove several new exact recurrence relations and an expansion of an 
exact formula for the inverse matrices in the previous equation in the 
results of Section \ref{Section_ExactRecFormulas}. 
In Section \ref{Section_SomeConjs} we readily 
computationally conjecture and prove that 
\begin{align*} 
s_{n,k}^{(-1)} & := \sum_{d|n} p(d-k) \mu(n / d)
\end{align*} 
where $p(n) = [q^n] (q; q)_{\infty}^{-1}$ denotes 
\emph{Euler's partition function}. 
This key conjecture immediately implies the results in 
Corollary \ref{cor_ExactFormulas_SpArithFns}. 
More precisely, the corollary provides exact finite divisor sum 
formulas for the special cases of \eqref{eqn_fn_matrix_eqn}
corresponding to the special arithmetic functions 
in \eqref{eqn_WellKnown_LamberSeries_Examples}. 

\subsection{Significance of our new results and conjectures} 

Questions involving divisors of an integer have been studied for millennia
and they underlie the deepest unsolved problems in number theory and related fields. 
The study of partitions, i.e., the ways to write a positive integer
as a sum of positive integers, is much younger, with Euler considered to be the
founder of the subject. The history of both subjects is rich and interesting
but in the interest of brevity we will not go into it here.

The two branches of number theory, additive and multiplicative, turn out
to be related in many interesting ways. Even though there are a number of 
important results connecting the theory of divisors with that of partitions, 
these are somewhat scattered in their approach. There seem to be many other 
connections, in particular in terms of different convolutions involving 
these functions, waiting to be discovered. We propose to continue the study 
of the relationship between divisors and partitions with the goal of identifying 
common threads and hopefully unifying the underlying theory. Moreover, 
it appears that, on the multiplicative number theory side, these connections 
can be extended to other important number theoretic functions such as 
Euler's totient function, Jordan's totient function, Liouville's function, 
the M\"obius function, and von Mangold's function, among others.

Our goal is to establish a unified global approach to studying the relationship between 
the additive and multiplicative sides of number theory. In particular, we hope 
to obtain a unified view of convolutions involving the partition function 
and number theoretic functions. To our knowledge, such an approach
has not been attempted yet. Convolutions have been used in nearly all areas of 
pure and applied mathematics. In a sense, they measure the overlap between two functions.
The idea for a unified approach for a large class of number theoretical
functions has its origin in Merca's article \cite{MERCA-LSFACTTHM} and in 
Schmidt's article \cite{SCHMIDT-LSFACTTHM}.

Perhaps our most interesting and important result, which we discovered 
computationally with \emph{Mathematica} and \emph{Maple} starting from an example formula 
given in the \emph{Online Encyclopedia of Integer Sequences} for the first column of the 
inverse matrices defined by \eqref{eqn_An_SquareMatrix_def} is stated in 
Theorem \ref{theorem_MainThm_InvMatrixDivSums}. 
The theorem provides an exact divisor sum formula for the 
inverse matrix entries, $s_{n,k}^{(-1)}$, involving a M\"obius transformation of the 
shifted Euler partition function, $p(n-k)$. 
This result is then employed to formulate new exact finite (divisor) sum formulas for 
each of the Lambert series functions, $a_n$, from the special cases in 
\eqref{eqn_WellKnown_LamberSeries_Examples}. 
These formulas are important since there are rarely such simple and universal identities 
expressing formulas for an entire class of special arithmetic functions considered in the 
context of so many applications in number theory and combinatorics. 
Generalizations, further applications, and topics for future research based on our 
work in this article are suggested in Section \ref{Section_Concl}. 

\section{Exact and recursive formulas for the inverse matrices} 
\label{Section_ExactRecFormulas}

\begin{prop}[Recursive Matrix-Product-Like Formulas] 
\label{prop_RecMatrix-Prod-Like_Formulas} 
We let $s_{i,j} := s_e(i, j) - s_o(i, j)$ denote the terms in the original 
matrices, $A_n$, from Schmidt's article and let $s_{i,j}^{(-1)}$ denote the 
corresponding entries in the inverse matrices, $A_n^{-1}$. 
Then we have that 
\begin{align*} 
s_{n,j}^{(-1)} & = - \sum_{k=1}^{n-j} s_{n,n+1-k}^{(-1)} \cdot s_{n+1-k,j} + 
     \delta_{n,j} \\ 
     & = 
     - \sum_{k=1}^{n-j} s_{n,n-k} \cdot s_{n-k,j}^{(-1)} + \delta_{n,j} \\ 
     & = 
     -\sum_{k=1}^{n} s_{n,k-1} \cdot s_{k-1,j}^{(-1)} + \delta_{n,j}. 
\end{align*} 
\end{prop} 
\begin{proof} 
The proof follows from the fact that for any $n \times n$ invertible matrices, 
$A_n$ and $A_n^{-1}$, with entries given in the notation above, we have the following 
inversion formula for all $1 \leq k,p \leq n$: 
\[
\sum_{j=1}^n s_{p,j} \cdot s_{j,k}^{(-1)} = \sum_{j=1}^n s_{p,j}^{(-1)} \cdot s_{j,k} = 
     \Iverson{p = k}. 
\] 
It is easy to see that the 
matrix, $A_n$, is lower triangular with ones on its diagonal for all $n \geq 1$ 
so that we may rearrange terms as in the formulas. We note that this property 
implies that $s_{i,j}^{(-1)} \equiv 0$ whenever $j < i$ by the 
adjoint (or adjugate) cofactor expansion for the inverse of a matrix. 
\end{proof} 

We can use the second of the two formulas given in the proposition 
repeatedly to obtain the 
following recursive, and then exact sums for the inverse matrix entries: 
\begin{align*} 
s_{i,j}^{(-1)} & = -\sum_{k=1}^i \sum_{k_2=1}^{k-1} \sum_{k_3=1}^{k_2-1} 
     s_{i,k-1} \cdot s_{k-1,k_2-1} \cdot s_{k_2-1,k_3-1} \cdot s_{k_3-1,j}^{(-1)} \\ 
     & \phantom{=\sum\sum } + 
     \sum_{k=j+2}^i s_{i,k-1} \cdot s_{k-1,j} - s_{i,j} + \delta_{i,j}. 
\end{align*} 
By inductively extending the expansions in the previous equation and noticing that the 
product terms in the multiple nested sums resulting from this procedure are 
eventually zero, we obtain the result in the next corollary.  

\begin{cor}[An Exact Nested Formula for the Inverse Matrices] 
\label{cor_exact_multsum_formula_for_sijinv} 
Let the notation for the next multiple, nested sums be defined as 
\begin{align*}
\Sigma_m(i, j) & := \underset{\text{$m$ total sums}}{\underbrace{\sum_{k_1=j+2}^i 
     \sum_{k_2=j+2}^{k_1-1} \cdots \sum_{k_m=j+2}^{k_{m-1}-1}}} 
     s_{i,k_1-1} \cdot s_{k_1-1,k_2-1} \times \cdots \times s_{k_m-1,j}. 
\end{align*} 
Then we may write an exact expansion for the inverse matrix entries as 
\begin{align*} 
s_{i,j}^{(-1)} & = \delta_{i,j} - s_{i,j} + \Sigma_1(i, j) - \Sigma_2(i, j) + 
     \cdots + (-1)^{i+j+1} \Sigma_{i-j}(i, j). 
\end{align*} 
\end{cor} 
\begin{proof} 
The proof is easily obtained by induction on $j$ and repeated applications of the 
third recurrence relation stated in 
Proposition \ref{prop_RecMatrix-Prod-Like_Formulas}. 
\end{proof} 

The terms in the multiple sums defined in the corollary 
are reminiscent of the formula for the 
multiplication of two or more matrices. 
We may thus potentially obtain statements of more productive exact results providing expansions 
of these inverse matrix terms by considering the nested, multiple sum formulas in 
Corollary \ref{cor_exact_multsum_formula_for_sijinv} as partial matrix products, 
though for the most part we leave the observation of such results as a topic for 
future investigation on these forms. 
However, given the likeness of the nested sums in the previous equations and in 
Corollary \ref{cor_exact_multsum_formula_for_sijinv} to 
sums over powers of the matrix $A_n$, we have computationally obtained the 
following related formula for the corresponding inverse matrices $A_n^{-1}$: 
\[
A_n^{-1} = \sum_{i=1}^{n-1} \binom{n-1}{i} (-1)^{i+1} A_n^{i-1},\ n \geq 2. 
\] 
We do not provide the proof by induction used to formally prove this 
identity here due to the complexity of the forms of the powers of the 
matrix $A_n$ which somewhat limit the utility of the formula at this point. 
We also notice that the corollary expresses the complicated inverse entry functions 
as a sum over products of sequences with known and comparatively simple generating 
functions stated in the introduction \cite[\cf Cor.\ 4.3]{MERCA-LSFACTTHM}. 
The results in Section \ref{Section_SomeConjs} 
provide a more exact representation of the entries of these inverse 
matrices for all $n$ obtained by a separate method of proof. 

\section{Some experimental conjectures} 
\label{Section_SomeConjs}

\begin{figure}[ht!]
\centering
\includegraphics[width=\textwidth]{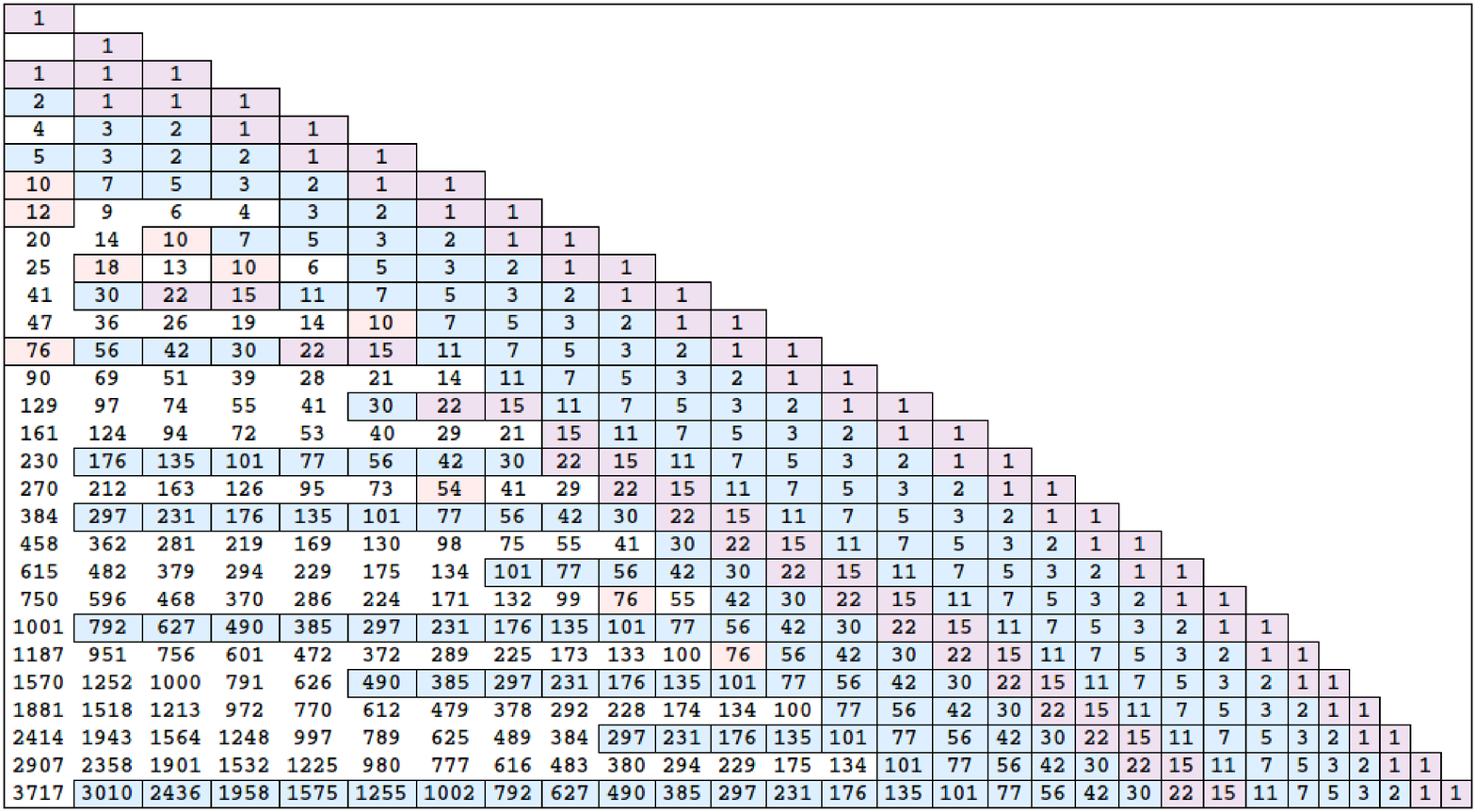}

\caption{The first $29$ rows of the function $s_{i,j}^{(-1)}$ where the values of 
         Euler's partition function $p(n)$ are highlighted in blue and the remaining 
         values of the partition function $q(n)$ are highlighted in 
         purple (in both sequences) or pink. } 
\label{figure_sijinv_first25_rows_highlighted} 

\end{figure} 

\subsection{Several figures and exact formulas} 

Based on our experimental analysis and some intuition with partition functions, 
we expect that the inverse matrix entries, $s_{i,j}^{(-1)}$, are deeply tied to the 
values of the Euler partition function $p(n)$. In fact, we are able to plot the 
first few rows and columns of the two-dimensional sequence in 
Figure \ref{figure_sijinv_first25_rows_highlighted} to obtain a highlighted listing 
of the values of special partition functions in the sequence of these matrix 
inverse entries. A quick search of the first few columns of the table in the 
figure turns up the following special entry in the online sequences database \cite{OEIS}. 

\begin{conjecture}[The First Column of the Inverse Matrices]
The first column of the inverse matrix is given by a convolution 
(dot product) of the partition function $p(n)$ and the 
M\"obius function $\mu(n)$ 
\cite[\href{https://oeis.org/A133732}{A133732}]{OEIS}. 
That is to say that 
\[ 
     s_{n,1}^{(-1)} = \sum_{d|n} p(d-1) \mu(n / d) 
     \quad\longmapsto\quad 
     \{1,0,1,2,4,5,10,12,20,25,41,47,\ldots\}, 
\]
i.e., so that by M\"obius inversion we have that 
\[ 
     p(n-1) = \sum_{d|n} s_{d,1}^{(-1)} 
     \quad\longmapsto\quad 
     \{1,1,2,3,5,7,11,15,22,\ldots\}. 
\]
\end{conjecture}

\begin{figure}[ht!]

\begin{minipage}{\linewidth}

\begin{minipage}{\linewidth} 
\begin{center} 
\tiny
\begin{equation*} 
\boxed{ 
\begin{array}{|c||cccccccccccc|} \hline 
 \mathbf{_n\setminus^k} & 1 & 2 & 3 & 4 & 5 & 6 & 7 & 8 & 9 & 10 & 11 & 12 \\ \hline\hline
 1 & 1 & 0 & 0 & 0 & 0 & 0 & 0 & 0 & 0 & 0 & 0 & 0 \\
 2 & 1 & 1 & 0 & 0 & 0 & 0 & 0 & 0 & 0 & 0 & 0 & 0 \\
 3 & 2 & 1 & 1 & 0 & 0 & 0 & 0 & 0 & 0 & 0 & 0 & 0 \\
 4 & 3 & 2 & 1 & 1 & 0 & 0 & 0 & 0 & 0 & 0 & 0 & 0 \\
 5 & 5 & 3 & 2 & 1 & 1 & 0 & 0 & 0 & 0 & 0 & 0 & 0 \\
 6 & 7 & 5 & 3 & 2 & 1 & 1 & 0 & 0 & 0 & 0 & 0 & 0 \\
 7 & 11 & 7 & 5 & 3 & 2 & 1 & 1 & 0 & 0 & 0 & 0 & 0 \\
 8 & 15 & 11 & 7 & 5 & 3 & 2 & 1 & 1 & 0 & 0 & 0 & 0 \\
 9 & 22 & 15 & 11 & 7 & 5 & 3 & 2 & 1 & 1 & 0 & 0 & 0 \\
 10 & 30 & 22 & 15 & 11 & 7 & 5 & 3 & 2 & 1 & 1 & 0 & 0 \\
 11 & 42 & 30 & 22 & 15 & 11 & 7 & 5 & 3 & 2 & 1 & 1 & 0 \\
 12 & 56 & 42 & 30 & 22 & 15 & 11 & 7 & 5 & 3 & 2 & 1 & 1 \\
 13 & 77 & 56 & 42 & 30 & 22 & 15 & 11 & 7 & 5 & 3 & 2 & 1 \\
 14 & 101 & 77 & 56 & 42 & 30 & 22 & 15 & 11 & 7 & 5 & 3 & 2 \\
 15 & 135 & 101 & 77 & 56 & 42 & 30 & 22 & 15 & 11 & 7 & 5 & 3 \\
 16 & 176 & 135 & 101 & 77 & 56 & 42 & 30 & 22 & 15 & 11 & 7 & 5 \\
 17 & 231 & 176 & 135 & 101 & 77 & 56 & 42 & 30 & 22 & 15 & 11 & 7 \\
 18 & 297 & 231 & 176 & 135 & 101 & 77 & 56 & 42 & 30 & 22 & 15 & 11 \\
 \hline
\end{array}
}
\end{equation*}
\end{center} 
\caption*{(i) The Divisor Sums $a_{n,k}^{\prime}$} 
\end{minipage} 

\begin{center} 
\tiny 
\begin{equation*} 
\boxed{ 
\begin{array}{|c||cccccccccccc|} \hline 
 \mathbf{_n\setminus^k} & 1 & 2 & 3 & 4 & 5 & 6 & 7 & 8 & 9 & 10 & 11 & 12 \\ \hline\hline
 1 & 1 & 0 & 0 & 0 & 0 & 0 & 0 & 0 & 0 & 0 & 0 & 0 \\
 2 & 0 & 1 & 0 & 0 & 0 & 0 & 0 & 0 & 0 & 0 & 0 & 0 \\
 3 & 1 & 1 & 1 & 0 & 0 & 0 & 0 & 0 & 0 & 0 & 0 & 0 \\
 4 & 2 & 1 & 1 & 1 & 0 & 0 & 0 & 0 & 0 & 0 & 0 & 0 \\
 5 & 4 & 3 & 2 & 1 & 1 & 0 & 0 & 0 & 0 & 0 & 0 & 0 \\
 6 & 5 & 3 & 2 & 2 & 1 & 1 & 0 & 0 & 0 & 0 & 0 & 0 \\
 7 & 10 & 7 & 5 & 3 & 2 & 1 & 1 & 0 & 0 & 0 & 0 & 0 \\
 8 & 12 & 9 & 6 & 4 & 3 & 2 & 1 & 1 & 0 & 0 & 0 & 0 \\
 9 & 20 & 14 & 10 & 7 & 5 & 3 & 2 & 1 & 1 & 0 & 0 & 0 \\
 10 & 25 & 18 & 13 & 10 & 6 & 5 & 3 & 2 & 1 & 1 & 0 & 0 \\
 11 & 41 & 30 & 22 & 15 & 11 & 7 & 5 & 3 & 2 & 1 & 1 & 0 \\
 12 & 47 & 36 & 26 & 19 & 14 & 10 & 7 & 5 & 3 & 2 & 1 & 1 \\
 13 & 76 & 56 & 42 & 30 & 22 & 15 & 11 & 7 & 5 & 3 & 2 & 1 \\
 14 & 90 & 69 & 51 & 39 & 28 & 21 & 14 & 11 & 7 & 5 & 3 & 2 \\
 15 & 129 & 97 & 74 & 55 & 41 & 30 & 22 & 15 & 11 & 7 & 5 & 3 \\
 16 & 161 & 124 & 94 & 72 & 53 & 40 & 29 & 21 & 15 & 11 & 7 & 5 \\
 17 & 230 & 176 & 135 & 101 & 77 & 56 & 42 & 30 & 22 & 15 & 11 & 7 \\
 18 & 270 & 212 & 163 & 126 & 95 & 73 & 54 & 41 & 29 & 22 & 15 & 11 \\
 \hline 
\end{array}
} 
\end{equation*} 
\end{center} 
\caption*{(ii) The Divisor Sums $a_{n,k}^{\prime\prime}$} 
\end{minipage}

\caption{A comparison of the two experimental divisor sum variants, 
         $a_{n,k}^{\prime}$ and $a_{n,k}^{\prime\prime}$, defined on 
         page \pageref{page_divisor_sum_variants_ankprime_ankprimeprime}. 
         Theorem \ref{conj_ExactFormulas_for_sijinv} 
         summarizes the results shown in these two sequence plots. } 
\label{figure_divisor_sum_plots_conj2} 

\end{figure} 

We are then able to explore further with the results from this first conjecture to 
build tables of the following two formulas involving our sequence, 
$s_{n,k}^{(-1)}$, and the shifted forms of the partition function, $p(n-k)$, where 
we take $p(n) \equiv 0$ when $n < 0$: 
\label{page_divisor_sum_variants_ankprime_ankprimeprime} 
\begin{align*} 
\tag{i} 
a_{n,k}^{\prime} & := \sum_{d|n} s_{d,k}^{(-1)} \\ 
\tag{ii} 
a_{n,k}^{\prime\prime} & := \sum_{d|n} p(d-k) \mu(n / d). 
\end{align*} 
The results of plotting these sequences for the first few rows and columns of 
$1 \leq n \leq 18$ and $1 \leq k \leq 12$, respectively, 
are found in the somewhat surprising and lucky results given in 
Figure \ref{figure_divisor_sum_plots_conj2}. 
From this experimental data, we arrive at the following second conjecture 
providing exact divisor sum formulas for the inverse matrix entries. 
The corollary immediately following this conjecture is implied by a correct 
proof of these results and from the formulas established in 
\cite[\S 3]{SCHMIDT-LSFACTTHM}. 

\begin{theorem}[Exact Formulas for the Inverse Matrices]
\label{conj_ExactFormulas_for_sijinv} 
\label{theorem_MainThm_InvMatrixDivSums}
For all $n , k \geq 1$ with $1 \leq k \leq n$, we have the following 
formula connecting the inverse matrices and the Euler partition 
function: 
\begin{align} 
\label{eqn_conj2_sijinv_exact_divsum_formula} 
s_{n,k}^{(-1)} & := \sum_{d|n} p(d-k) \mu(n / d). 
\end{align}
\end{theorem} 
\begin{proof} 
We see that the first equation in \eqref{eqn_conj2_sijinv_exact_divsum_formula} 
which we seek to prove is equivalent to 
\begin{align*} 
\notag 
p(n-k) & := \sum_{d|n} s_{d,k}^{(-1)}. 
\end{align*} 
We next consider the variant of the Lambert series factorization theorem in 
\eqref{eqn_Merca_LSFactorizationThm} applied to the Lambert series in 
\eqref{eqn_LambertSeriesfb_def} with $a_n := s_{n,k}^{(-1)}$ for a 
fixed integer $k \geq 1$. 
In particular, the identity in \eqref{eqn_Merca_LSFactorizationThm} implies that 
\begin{align*} 
\sum_{d|n} s_{d,k}^{(-1)} & = \sum_{m=0}^n \sum_{j=1}^{n-m} \left( 
     s_o(n-m, j) - s_e(n-m, j)\right) s_{j,k}^{(-1)} \cdot p(m) \\ 
     & = 
     \sum_{m=0}^n \delta_{n-k,m} \cdot p(m) \\ 
     & = 
     p(n-k), 
\end{align*} 
where we have by our matrix formulation in \eqref{eqn_An_SquareMatrix_def} that 
\begin{align*} 
\sum_{j=1}^{m} \left(s_o(m, j) - s_e(m, j)\right) s_{j,k}^{(-1)} & = 
     \delta_{m,j}. 
\end{align*} 
Thus by M\"obius inversion, we have our key formula for the inverse matrix 
entries given in \eqref{eqn_conj2_sijinv_exact_divsum_formula}. 
\end{proof} 

We notice that the last equation given in the conjecture implies that we have a 
Lambert series generating function for the inverse matrix entries given by 
\begin{align*} 
\sum_{n \geq 1} \frac{s_{n,k}^{(-1)} q^n}{1-q^n} & = \frac{q^k}{(q; q)_{\infty}}. 
\end{align*} 
for fixed integers $k \geq 1$. 
We also note that where Merca's article \cite{MERCA-LSFACTTHM} provides the 
partition function representation for the sequence $s_{n,k}$ in the matrix 
interpretation established in \cite{SCHMIDT-LSFACTTHM}, 
the result in the theorem above effectively provides us with an exact identity for the 
corresponding sequence of inverse matrix entries, $s_{n,k}^{(-1)}$, employed as in 
Schmidt's article to obtain the new expressions for several key special multiplicative 
functions. 

One important and interesting consequence of the result in 
Theorem \ref{conj_ExactFormulas_for_sijinv} is that we have now 
completely specified several new formulas which provide exact representations for a 
number of classical and special multiplicative functions cited as examples in 
\eqref{eqn_WellKnown_LamberSeries_Examples} of the introduction. 
These formulas, which are each expanded in the 
next corollary, connect the expansions of several special multiplicative functions to 
sums over divisors of $n$ involving Euler's partition function $p(n)$. 
In particular, we can now state several specific identities for classical number theoretic 
functions which connect the seemingly disparate branches of multiplicative number theory 
with the additive nature of the theory of partitions and special partition functions. 
The results in the next corollary are expanded in the following forms: 

\begin{cor}[Exact Formulas for Special Arithmetic Functions] 
\label{cor_ExactFormulas_SpArithFns} 
For natural numbers $m \geq 0$, let the next component sequences 
defined in \cite[\S 3]{SCHMIDT-LSFACTTHM} be defined by the formulas 
\begin{align*} 
B_{\phi,m} & = m+1 - \frac{1}{8}\Biggl(8 - 5 \cdot (-1)^{u_1} - 4 \left( 
     -2 + (-1)^{u_1} + (-1)^{u_2}\right) m \\ 
     & \phantom{=m+1 - \frac{1}{8}\Biggl(8\ } + 
     2 (-1)^{u_1} u_1 (3u_1+2) + 
     (-1)^{u_2} (6u_2^2+8u_2-3)\Biggr) \\ 
B_{\mu, m} & = \Iverson{m = 0} + \sum_{b = \pm 1} 
     \sum_{k=1}^{\lfloor \frac{\sqrt{24m+25}-b}{6} \rfloor} 
     (-1)^k \Iverson{m+1-k(3k+b)/2 = 1} \\ 
B_{\lambda, m} & = 
     \Iverson{\sqrt{m+1} \in \mathbb{Z}} - \sum_{b = \pm 1} 
     \sum_{k=1}^{\lfloor \frac{\sqrt{24m+1}-b}{6} \rfloor} 
     (-1)^{k+1} \Iverson{\sqrt{m+1-k(3k+b)/2} \in \mathbb{Z}} \\ 
B_{\Lambda, m} & = 
     \log(m+1) - \sum_{b = \pm 1} 
     \sum_{k=1}^{\lfloor \frac{\sqrt{24m+1}-b}{6} \rfloor} 
     (-1)^{k+1} \log(m+1-k(3k+b)/2) \\ 
B_{|\mu|, m} & = 
     2^{\omega(m+1)} - \sum_{b = \pm 1} 
     \sum_{k=1}^{\lfloor \frac{\sqrt{24m+1}-b}{6} \rfloor} 
     (-1)^{k+1} 2^{\omega(m+1-k(3k+b)/2)} \\ 
B_{J_t, m} & = 
     (m+1)^t - \sum_{b = \pm 1} 
     \sum_{k=1}^{\lfloor \frac{\sqrt{24m+1}-b}{6} \rfloor} 
     (-1)^{k+1} (m+1-k(3k+b)/2)^t, 
\end{align*} 
where $u_1 \equiv u_1(m) := \lfloor (\sqrt{24m+1}+1)/6 \rfloor$ and 
$u_2 \equiv u_2(m) := \lfloor (\sqrt{24m+1}-1)/6 \rfloor$. 
Then we have that 
\begin{align*} 
\phi(n) & = \sum_{m=0}^{n-1} \sum_{d|n} p(d-m-1) \mu(n / d) B_{\phi,m} \\ 
\mu(n) & = \sum_{m=0}^{n-1} \sum_{d|n} p(d-m-1) \mu(n / d) B_{\mu,m} \\ 
\lambda(n) & = \sum_{m=0}^{n-1} \sum_{d|n} p(d-m-1) \mu(n / d) B_{\lambda,m} \\ 
\Lambda(n) & = \sum_{m=0}^{n-1} \sum_{d|n} p(d-m-1) \mu(n / d) B_{\Lambda,m} \\ 
|\mu(n)| & = \sum_{m=0}^{n-1} \sum_{d|n} p(d-m-1) \mu(n / d) B_{|\mu|,m} \\ 
J_t(n) & = \sum_{m=0}^{n-1} \sum_{d|n} p(d-m-1) \mu(n / d) B_{J_t,m}. 
\end{align*} 
The corresponding formulas for the average orders, $\Sigma_{a,x}$, of these special 
arithmetic functions are obtained in an initial form by summing the 
right-hand-sides of the previous equations over all $n \leq x$. 
\end{cor} 

We can also compare the results of the recurrence relations in the previous 
corollary to two other identical statements of these results. 
In particular, if we define the sequence 
$\{G_j\}_{j \geq 0} = \{0,1,2,5,7,12,15,22,26,35,40,51,\ldots\}$ as in 
\cite[\S 1]{MERCA-LSFACTTHM} by the formula 
\[
G_j = \frac{1}{2} \left\lceil \frac{j}{2} \right\rceil \left\lceil \frac{3j+1}{2} \right\rceil, 
\]
then by performing a divisor sum over $n$ in the previous equations, we see that the 
sequence pairs in the form of \eqref{eqn_LambertSeriesfb_def} satisfy 
\[
b_n = \sum_{k=1}^n \sum_{j=0}^{k-1} p(n-k) (-1)^{\lceil j/2 \rceil} b(k - G_j). 
\] 
We immediately notice the similarity of the recurrence relation for $b_n$ given in the 
last equation to the known result from \cite[Thm.\ 1.4]{SCHMIDT-LSFACTTHM} 
which states that 
\[
b_n = \sum_{j=0}^n (-1)^{\lceil j/2 \rceil} b_{n-G_j}, 
\] 
and which was proved by a separate non-experimental approach in the reference. 

\begin{remark}[An Experimental Conjecture]  
Since we have a well-known recurrence relation for the partition function 
given by 
\[ 
     p(n) = \sum_{k=1}^{n} (-1)^{k+1}\left(p(n-k(3k-1)/2)+p(n-k(3k+1)/2)\right), 
\] 
we attempt to formulate an analogous formula for the $s_{i,j}^{(-1)}$ using 
\eqref{eqn_conj2_sijinv_exact_divsum_formula}, which leads us to the sums 
\begin{align*} 
a_{n}^{\prime\prime\prime} & := \sum_{k=1}^n (-1)^{k+1}\left( 
     s_{n,k(3k-1)/2}^{(-1)} + s_{n,k(3k+1)/2}^{(-1)} 
     \right) \\ 
     & \quad\longmapsto\quad 
     \{1,1,2,3,6,7,14,17,27,34,55,63,\ldots\}. 
\end{align*} 
A search in the integer sequences database suggests that this sequence denotes 
the number of partitions of $n$ into relatively prime parts, or alternately, 
aperiodic partitions of $n$ \cite[\href{https://oeis.org/A000837}{A000837}]{OEIS}. 
We notice the additional, and somewhat obvious and less interesting, 
identity which follows from the recurrence relation for $p(n)$ 
given above expanded in the form of 
\[
\sum_{k=0}^n (-1)^{\lceil k/2 \rceil} s_{n,G_k}^{(-1)} = 0. 
\]
\end{remark} 

\subsection{Other properties related to the partition function} 

\begin{prop}[Partition Function Subsequences] 
Let $n$ be a positive integer. For $\lceil n/2 \rceil < k \leq n$,
$$s_{n,k}^{(-1)} = p(n-k).$$ 
The indices of the first few rows such that $$s_{n,k}^{(-1)} = p(n-k)$$ is 
true for all $1 < k \leq n$ are $\{2, 3, 5, 7, 11, 13, 17, 19, 23, 29, \ldots \}$. 
\end{prop}
\begin{proof} 
This result is immediate from the divisor sum in 
\eqref{eqn_conj2_sijinv_exact_divsum_formula} where the only 
divisor of $n$ in the range $\lceil n/2 \rceil < k \leq n$ is $n$ itself.
\end{proof} 

\begin{prop}[Partition Function Subsequences for Prime $n$] 
For $n$ prime and $1\leq k \leq n$,
	$$s_{n,k}^{(-1)} = p(n-k) - \delta_{1,k},$$ 
where $\delta_{i,j}$ is the Kronecker delta function.
\end{prop}
\begin{proof} 
This result is also immediate from the divisor sum in 
\eqref{eqn_conj2_sijinv_exact_divsum_formula} where the only 
divisors of the prime $n$ are $1$ and $n$ and $p(1-k) = \delta_{k,1}$ 
by convention. In particular, we have that 
\[ 
     s_{n,k}^{(-1)} = \mu(p) p(1-k) + \mu(1) p(n-k), 
\] 
for all $1 \leq k \leq n$. 
\end{proof} 


The next two results which we initially obtained experimentally from tables of the 
matrix inverse entries follow along the same lines as the previous two propositions. 
Given the ease with which we proved the last formulas for prime $n$, we omit the 
one-line proofs of the next two results below. Note that by the formula in 
\eqref{eqn_conj2_sijinv_exact_divsum_formula}, we may also strengthen these results 
to prime powers of the form $n = p^k$ for $k \geq 1$ and any prime $p$. 

\begin{prop} 
For $n$ prime,
	$$s_{n^2,k}^{(-1)} =
	\begin{cases}
	p(n^2-k) - p(n-k), & \text{for $1\leq k\leq n$,}\\
	p(n^2-k), & \text{for $n<k\leq n^2$.}
	\end{cases}
	$$ 
\end{prop}

\begin{prop} For $n$ prime,
	$$s_{2n,k}^{(-1)} =
	\begin{cases}
	p(2n-k) - p(n-k)-p(2-k)+\delta_{1,k}, & \text{for $1\leq k\leq 2$,}\\
	p(2n-k) - p(n-k), & \text{for $2 < k\leq n$,}\\
	p(2n-k), & \text{for $n<k\leq 2n$.}
	\end{cases}
	$$ 
\end{prop}

A similar argument to the above can be used to show that if $q,r \in \mathbb{Z}^{+}$ are 
relatively prime positive integers, then we have that 
\[
s_{qr,k}^{(-1)} = \delta_{1,k} - p(q-k) - p(r-k) + p(qr-k), 
\]
which as we observe is another example of an additive formula we have obtained 
defining an inherently multplicative structure in terms of additive functions. 
Notably, we can use this observation to show that if the 
arithmetic function $a_n$ in \eqref{eqn_LambertSeriesfb_def} is multiplicative, 
then we have that $a_q \cdot a_r = b_{qr} - b_p - b_q + b_1$ 
for all positive integers $p,q$ such that $(p, q) = 1$. 
We can then form subsequent generalizations for products of pairwise relatively prime 
integers, $q_1,q_2,\ldots,q_m$, accordingly. 

\section{Conclusions} 
\label{Section_Concl}

\subsection{Summary} 

We have proved a unified form of the Lambert series factorization theorems from the 
references \cite{MERCA-LSFACTTHM,SCHMIDT-LSFACTTHM} which allows us to exactly express 
matrix equations between the implicit arithmetic sequences, $a_n$ and $b_n$, in 
\eqref{eqn_LambertSeriesfb_def} and in the classical special cases in 
\eqref{eqn_WellKnown_LamberSeries_Examples}. 
More precisely, we have noticed that the invertible matrices, $A_n$, from Schmidt's article are 
expressed through the factorization theorem in \eqref{eqn_Merca_LSFactorizationThm} 
proved by Merca. We then proved new divisor sum formulas involving the partition function $p(n)$ 
for the corresponding inverse matrices which 
define the sequences, $a_n$, in terms of only these matrix entries and the secondary sequence of 
$b_n$ as in the results from \cite{SCHMIDT-LSFACTTHM}. 

The primary application of our new matrix formula results is stated in 
Corollary \ref{cor_ExactFormulas_SpArithFns}. 
The corollary provides new exact finite (divisor) sum formulas for the 
special arithmetic functions, $\phi(n)$, $\mu(n)$,  $\lambda(n)$, 
$\Lambda(n)$, $|\mu(n)|$, and $J_t(n)$, and the corresponding 
partial sums defining the average orders of these functions. 
One related result not explicitly stated in Schmidt's article provides a discrete 
(i.e., non-divisor-sum) convolution for the average order of the sum-of-divisors function, 
denoted by $\Sigma_{\sigma,x} := \sum_{n \leq x} \sigma(n)$, in the form of 
\cite[\S 3.1]{SCHMIDT-LSFACTTHM} 
\begin{align*} 
\Sigma_{\sigma,x+1} & = \sum_{s = \pm 1} \left( \sum_{0 \leq n \leq x} 
     \sum_{k=1}^{\left\lfloor \frac{\sqrt{24n+25}-s}{6} \right\rfloor} 
     (-1)^{k+1} \frac{k(3k+s)}{2} \cdot p(x-n)\right). 
\end{align*} 
Other related divisor sum results that can be stated in terms of our new inverse matrix 
formulas implied by Theorem \ref{theorem_MainThm_InvMatrixDivSums} are found, 
for example, in Merca's article \cite[\S 5]{MERCA-LSFACTTHM}. 

\subsection{Generalizations} 

Merca showed another variant of the Lambert series factorization theorem 
stated in the form of \cite[Cor.\ 6.1]{MERCA-LSFACTTHM}
\begin{align*} 
\sum_{n \geq 1} \frac{a_n q^{2n}}{1-q^n} & = \frac{1}{(q; q)_{\infty}} 
     \sum_{n \geq 1} \sum_{k=1}^{\lfloor n/2 \rfloor} \left( 
     s_o(n-k, k) - s_e(n-k, k)\right) a_k \cdot q^n. 
\end{align*} 
If we consider the generalized Lambert series formed by taking derivatives of 
\eqref{eqn_LambertSeriesfb_def} from \cite{SCHMIDT-COMBSUMSBDDDIV} in the 
context of finding new relations between the generalized sum-of-divisors functions, 
$\sigma_{\alpha}(n)$, we can similarly formulate new, alternate forms of the 
factorization theorems unified by this article. 
For example, suppose that $k, m \geq 0$ are integers and 
consider the factorization theorem resulting from analysis of the sums 
\begin{align*}  
\sum_{n \geq 1} \frac{a_n q^{(m+1)n}}{(1-q^n)^{k+1}} & = \frac{1}{(q; q)_{\infty}} 
     \sum_{n \geq 1} \sum_{i=1}^{\lfloor n/(m+1) \rfloor} s_{n-m, i} 
     \frac{a_i}{(1-q^i)^k} \cdot q^n,  
\end{align*} 
so that we have the factorization theorem providing that the previous series 
are expanded by 
\begin{align*}  
\sum_{n \geq 1} \frac{a_n q^{(m+1)n}}{(1-q^n)^{k+1}} & = \frac{1}{(q; q)_{\infty}} 
     \sum_{n \geq 1} \sum_{i=1}^{\left\lfloor \frac{n}{m+1} \right\rfloor} 
     \sum_{j=0}^{\left\lfloor \frac{n-m}{i} \right\rfloor} 
     \binom{k-1+j}{k-1} s_{n-m-ji, i} \cdot a_i \cdot q^n, 
\end{align*} 
and so that when $m \geq k$ the series coefficients of these modified Lambert series 
generating functions are given by 
\begin{align*} 
\sum_{\substack{d|n \\ d \leq \left\lfloor \frac{n}{m+1} \right\rfloor}} 
     \scriptstyle{\binom{\frac{n}{d}-1-m+k}{k}} a_d & = 
     \sum_{q=0}^n \sum_{i=1}^{\left\lfloor \frac{n-q}{m+1} \right\rfloor} 
     \sum_{j=0}^{\left\lfloor \frac{n-q-m}{i} \right\rfloor} 
     \binom{k-1+j}{k-1} s_{n-q-m-ji, i} \cdot a_i \cdot p(q). 
\end{align*} 
Thus, again, as in Merca's article, the applications and results in 
Corollary \ref{cor_ExactFormulas_SpArithFns} 
can be repeated in the context of a slightly different 
motivation for considering these factorization theorems. 

\subsection{Topics for future research} 

Topics for future research based on the unified factorization theorem results we have 
proved within the article include investigating the properties of the generalizations 
defined in the last subsection, considering congruences for the partition function and the 
inverse matrix entries, $s_{n,k}^{(-1)}$, and finding useful new asymptotic formulas for the 
average orders of the special functions in Corollary \ref{cor_ExactFormulas_SpArithFns}. 

The last topic is of particular interest since we have given an explicit formula for the 
M\"obius function, $\mu(n)$, which holds for all $n \geq 0$. The problem of determining 
whether the average order, $M(x) := \sum_{n \leq x} \mu(n)$, 
of this particular special function is bounded by 
$M(x) = O(x^{1/2+\varepsilon})$ for all sufficiently small $\varepsilon > 0$ is equivalent to the 
Riemann hypothesis. In light of the significance of this problem, we must at least suggest 
our approach towards formulating new exact, non-approximate properties of this average order 
sequence for all $x \geq 1$. 

\subsection*{Acknowledgments} 

The authors thank the referees for their helpful insights and comments on 
preparing the manuscript.

\end{document}